\documentstyle[amssymb]{amsart}
\newcommand{\lra}{\longrightarrow}
\newcommand{\bi}{\begin{itemize}}
\newcommand{\ei}{\end{itemize}}
\newcommand{\bpf}{\begin{proof}}
\newcommand{\epf}{\end{proof}}
\newcommand{\be}{\begin{equation}}
\newcommand{\ee}{\end{equation}}
\newcommand{\bl}{\begin{lemma}}
\newcommand{\el}{\end{lemma}}
\newcommand{\bt}{\begin{theorem}}
\newcommand{\et}{\end{theorem}}

\newcommand{\ann}{\mathrm{ann}}

\newcommand{\Kdim}{\mathrm{K.dim}}

\newcommand{\al}{\alpha}

\newcommand{\gb}{\beta}
\newcommand{\gam}{\gamma}

\newcommand{\gl}{\lambda}
\newcommand{\fsl}{\mathfrak{sl}}

\newcommand{\C}{\mathbb{C}}
\newcommand{\Z}{\mathbb{Z}}
\newcommand{\N}{\mathbb{N}}

\newtheorem{theorem}{Theorem}[section]
 \newtheorem{lemma}[theorem]{Lemma}

\newtheorem{thm}{Theorem}[section]
\newtheorem*{thm*}{Theorem}
\newtheorem{lem}[thm]{Lemma}
\newtheorem*{lem*}{Lemma}
\newtheorem{prop}[thm]{Proposition}
\newtheorem*{prop*}{Proposition}

\newtheorem*{cor*}{Corollary}

\newtheorem*{defn*}{Definition}

\newtheorem*{notadefn*}{Notation and Definition}

\newtheorem*{nota*}{Notation}

\newtheorem*{note*}{Remark}
\newtheorem*{notes*}{Remarks}

\newtheorem*{ex*}{Example}

\newenvironment{proof}{\par\noindent{\bf Proof.}}{$\square$\par\bigskip}

\begin{document}

\title[Monolithic modules over Noetherian Rings]
      {\Large \bf Monolithic modules over Noetherian Rings}
\author[Paula A.A.B. Carvalho and Ian M. Musson]
       {Paula A.A.B. Carvalho \\ Ian M. Musson }
\thanks{The first author was partially supported by Centro de Matem\'atica da Universidade do Porto, financed by FCT
(Portugal) through the programs POCTI and POSI with national and
European community structural funds.}
\address{Departamento de Matem\'atica\\
         Faculdade de Ci\^encias\\
         Universidade do Porto\\
         Rua do Campo Alegre 687, 4169-007 Porto, Portugal}
\email{pbcarval@@fc.up.pt}
\address{University of Wisconsin-Milwaukee\\
         PO Box 413\\
         Milwaukee, WI 53201 U.S.A.}
\email{musson@@csd.uwm.edu}

\subjclass{16E70, 16P40}
\date\today

\begin{abstract}
We study finiteness conditions on essential extensions of simple
modules over the quantum plane,  the quantized Weyl algebra and Noetherian down-up
algebras. The results achieved improve the ones obtained in
\cite{CarvalhoLompPusat} for down-up algebras.
\end{abstract}
\maketitle

\section{Introduction}
In this paper we consider the following property of a Noetherian ring $A$:
\begin{center}$(\diamond)\:\:$ Injective hulls of simple left $A$-modules are locally Artinian.\end{center}
Property $(\diamond)$ has an interesting history.  Indeed it was
shown by A.V. Jategaonkar \cite{Jategaonkar74} and J.E. Roseblade
\cite{Roseblade} that if $G$ is a polycylic-by-finite  group, then
the group ring $RG$ has property $(\diamond)$ whenever $R$ is the
ring of integers, or is a field that is algebraic over
a finite field see also \cite{Passman85} Section 12.2.  This
result is the key step in the positive solution of a problem of P.
Hall, \cite{Hall}. P. Hall asked whether every finitely generated
abelian-by-(polycylic-by-finite) group is residually finite. In
\cite{Roseblade} a module $M$ is called {\it
monolithic} if it has a unique minimal submodule. %$L$.  If this is the case, then $L$ is called the {\it lith} of $M$.
  Note that $A$
has property $(\diamond)$ if and only if every finitely generated
monolithic $A$-module is Artinian.  We have revived the older,
shorter terminology in the title of this paper. A.V. Jategaonkar
showed in \cite{Jategaonkar_conj} that a fully bounded Noetherian
ring $R$ satisfies property $(\diamond)$, and used this fact to
show that
Jacobson's conjecture holds for $R$. %,  that the intersection of the powers of the Jacobson radical of $R$ is zero.

Returning to the group ring situation, suppose $G$ is a polycylic-by-finite  group, $K$  is a field, $A = KG$
 and $E$ is the injective hull of a finite-dimensional $A$-module.
It was shown by K.A. Brown, \cite{Ken81} that if $K$ has
characteristic zero, then $E$ is locally finite dimensional, and
this fact and some  Hopf algebra theory was used by S. Donkin to
show that $E$ is in fact Artinian \cite{Donkin81}. Note that
injective comodules over coalgebras are always locally finite
dimensional. Similar results were obtained when  $K$ has positive
characteristic by the second author \cite{Musson80} using methods
that more closely follow the argument used for commutative rings
in \cite{SV}.

The first examples of Noetherian rings for which property
$(\diamond)$ does not hold were given by the second author for
group algebras and enveloping algebras, see \cite{M},
\cite{Musson_Counter} and \cite[Example 7.15]{CH}. On the other
hand R.P. Dahlberg \cite{Dahlberg89} showed that injective hulls
of simple modules over $U({\frak sl}_2)$ are locally Artinian.

Interest in property $(\diamond)$ was renewed recently by a
question of P.F. Smith. Smith asked whether Noetherian down-up
algebras have property $(\diamond).$ Given a field $K$ and
$\al,\gb,\gam$ arbitrary elements of $K$, the associative algebra
$A=A(\al,\gb,\gam)$ over $K$ with generators $d,u$ and defining
relations
               $$(R1)\qquad d^2u=\al dud+\gb ud^2+\gam d$$
               $$(R2)\qquad du^2=\al udu+\gb u^2d+\gam u$$
is called a down-up algebra. Down-up algebras were introduced by G. Benkart and
T. Roby \cite{BenkartRoby}. %, and have received considerable attention ...
In \cite{KirkmanMussonPassman} it is shown that $A(\al,\gb,\gam)$ is Noetherian if and only if $\gb\neq 0$.
Some examples of down-up algebras with property $(\diamond)$ were given in \cite{CarvalhoLompPusat}.
In this paper we study Noetherian down-up
algebras having property $(\diamond),$ and in particular we exhibit the first examples that do not have this property.
These examples are constructed using the fact that when $\gam = 0$, (resp. $\gam = 1$) the quantum plane,  (resp. the quantized Weyl algebra) is an image of $A.$

An interesting class of down-up algebras arises in the following way. For $\eta\neq 0$, let $A_{\eta}$ be the algebra with generators $h, e, f$ and relations
$$he-eh=e,$$
$$hf-fh=-f,$$
$$ef-\eta fe=h.$$
Then $A_{\eta}$ is isomorphic to a down-up algebra $A(1+\eta, -\eta, 1)$ and conversely any down-up algebra $A(\alpha, \beta, \gamma)$
with $\beta\neq 0\neq \gamma$ and $\alpha+\beta=1$ has the above form. Note that $A_1\simeq U({\fsl}(2))$ and $A_{-1}\simeq U(osp(1,2))$.
When $\eta$ is not a root of unity, we have been unable to determine whether property $(\diamond)$ holds. However we resolve the issue in all other cases.
Our main result is as follows.
\begin{theorem} \label{c2}Suppose that  $A= A(\alpha, \beta, \gamma)$ is  a Noetherian down-up algebra, and assume that if $\al +\gb = 1,$ and $\gam \neq 0,$ then $\gb$ is a root of unity.
Then any finitely generated monolithic $A$-module is Artinian if
and only if the roots of $X^2-\alpha X-\beta$ are roots of unity.
\end{theorem}

We remark that a characterization of property
$(\diamond)$ for Noetherian rings remains rather elusive.  Even  a comparison of the examples for the quantum plane and quantized Weyl algebra does not seem easy to make,
%\ref{crqp} and
see Section \ref{qwa} for further remarks.  Thus it seems worthwhile to study examples of rings with low GK-dimension, and down-up algebras provide an interesting test-case for property
$(\diamond)$.  Much current research in non-commutative algebraic geometry also centers on low dimensional algebras, and in particular down-up algebras are studied as non-commutative threefolds by Kulkarni in
\cite{K2008}.

We thank Kenny Brown for his comments on a preliminary version of this paper.
\section{Preliminaries.} {\label{I201}} If $r \in K$ and $x, y$ are elements of a $K$-algebra we set
$[x,y]_r  = xy - ryx.$  Throughout this paper we will assume the equation $0=\lambda^2-\al\lambda-\gb$ has roots $r,s \in K$. Suppose $q \in K$ is nonzero and consider the algebra $B(q) = K[a,b]$ generated by $a, b$ subject to the relation $ab = qba.$  In addition let $C(q) = K[a,b]$ denote the algebra generated by $a, b$ subject to the relation $ab - qba =1.$
The algebras  $B(q), C(q)$ are known as the {\it coordinate algebra of the quantum plane} and the  {\it quantized Weyl algebra} respectively.
\bl \label{homomorphic}  \begin{itemize}\item[{}]
\item[(a)] The algebra $B(r)$ is a homomorphic image of $A=A(\al,\gb,0)$. \item[(b)] If $s\neq 1$ the algebra $C(r)$ is a homomorphic image of $A=A(\al,\gb,1)$.
\end{itemize}
\el \bpf  \; If $\gam = 0,$ relations (R1) and (R2) can be written in the form
$$[d,[d,u]_r]_s = [[d,u]_r,u]_s = 0.$$ Thus both relations follow from the relation $[d,u]_r = 0,$ so there is a map from  $A=A(\al,\gb,\gam)$ onto $B(r)$ sending $d$ to $a$ and $u$ to $b$.\\
\\
On the other hand if $\gam\neq 0$, we can assume $\gamma =1$. If $s\neq 1$, let $t\in K$ be such that $t(s-1)=1$. Relations $(R1)$ and $(R2)$ can now be written in the form
$$[d,[d,u]_r -t]_s = [[d,u]_r-t,u]_s = 0.$$
Since $[ta,b]_r-t=0$ in $C(r)$, there is an homomorphism from $A$ onto $C(r)$ sending $d$ to $ta$ and $u$ to $b$.
\epf
The above Lemma will be used, together with the results of the next two subsections, to produce examples of down-up algebras that do not satisfy property $(\diamond)$.  Note however that if exactly one of the roots of the Equation $X^2-\alpha X-\beta$ is equal to 1, the Lemma tells us only that the first Weyl algebra is  a homomorphic image of $A=A(\al,\gb,1)$. In  this case the Lemma is of no use in constructing counterexamples.
\section{The Coordinate Ring of the Quantum Plane.}\label{crqp}
If  $q$ is an element of $K$ which is not a root of unity we show that $B=B(q)$ does not satisfy property $(\diamond)$. Consider the left ideals $I = B(ab-1)(a-1) \subset J = B(a-1)$, and set $M = B/I, V = J/I$ and $W = B/J.$ Then there is an exact sequence $$0 \lra V \lra M \lra W \lra 0.$$
\bt \label{theoremB} \begin{itemize}\item[{}]
\item[{(a)}]  The module $M$ is a non-Artinian essential extension of the simple submodule $V$.  \item[{(b)}] The submodules of $W$ are linearly ordered by inclusion, and are pairwise non-isomorphic.
\end{itemize}
 \et \bpf \;
 {\it Step 1: V is simple.} Clearly $V$ is generated by the element $v_0 = (a-1)+I$.  For $n\geq 0,$ set
 $$v_n = b^nv_0, \quad v_{-n} = a^nv_0.$$ Then using $abv_0 = v_0,$ we obtain for all $n \geq 0$,
\begin{equation}
av_{n+1} = q^nv_n, \quad bv_{-n-1} = q^{-n-1}v_{-n}.
\label{eq1}
\end{equation}
\\
%PAULA Do you know how to number Equations? I dont know why I get errors if %I do not comment out the next line.
% \be \label{eq1}  av_{n+1} = q^nv_n, \quad bv_{-n-1} = q^{-n-1}v_{-n}.\ee
Furthermore for all integers $n,$
% \be \label{eq2}  abv_{n} = q^nv_n.\ee
\begin{equation}
abv_{n} = q^nv_n.
\label{eq2}
\end{equation}
It is easy to see that $V$ is spanned by the set $X = \{v_n|n\in \mathbb{Z}\}$, and it follows from  equation  (\ref{eq2}) that the set $X$ is linearly independent. Equation  (\ref{eq2}) also implies that any submodule of $V$ is spanned by a subset of $X.$  Then simplicity of $V$ follows from equation  (\ref{eq1}). \\
\\
{\it Step 2: Proof of (b).} Clearly $W$ is generated by the element $w_0 = 1+J$ and spanned over $K$ by the set $Y= \{w_n|n\geq 0\},$
where  $w_n = b^nw_0$.
Furthermore for all $n \geq 0,$
% \be \label{eq3}  aw_{n} = q^nw_n.\ee
\begin{equation}\label{eq2*}
aw_{n} = q^nw_n.
\end{equation}
As in the proof of Step 1, $Y$ is linearly independent. Equation  (\ref{eq2*}) also implies that any submodule of $W$ is spanned by a subset of $Y.$ Now for all $n \geq 0$ set
$$W_n = span\{w_m|m\geq n\} = Bw_n.$$ Consideration of the action of $b$ now shows that a complete list of non-zero submodules of $W$ is
$$W = W_0 \supset W_1 \supset W_2 \ldots.$$ To  complete the proof of (b) we observe that $a$ acts as multiplication by $q^n$ on  the unique simple quotient of $W_n.$\\
\\
{\it Step 3: There is no element $v \in V$ such that $(a-q^m)v = v_m.$} If $v \in V$ is non-zero we can write $v$ as a linear combination of basis elements, $v = \sum_{i=r}^s \gl_iv_i,$ where $\gl_r, \gl_s$ are nonzero.  Then we set $|v| = s-r.$  From equations (\ref{eq1}), it follows that $|(a-q^m)v| = s-r+1.$  Clearly this gives the assertion.  \\
\\
{\it Step 4: Proof of (a).} Set $m_n = b^n + I$ for $n \geq 0.$ Then $m_n$ maps onto $w_n$ under the natural map $M \lra W.$  Thus the set  $\{v_n, m_p|n, p \in {\mathbb{Z}}, n \ge 0\}$ is a basis for $M.$ Since $am_0 = m_0 + v_0,$ it
follows that \begin{eqnarray}\label{}
am_n &=& q^nb^nam_0\nonumber \\ &=& q^n(m_n+v_n).\nonumber
\end{eqnarray}
Suppose that $m = \sum_{i \in I} \gl_im_i +v$ is a nonzero element of $M.$ We assume that $v \in V$, $|I|$ is non-empty, and that $\gl_i$ is a non-zero scalar for all $i \in I.$  Then we show by induction on $|I|$ that $Bm \cap V$ is non-zero.   Suppose that $n \in I$, and without loss that $\gl_n = 1.$  If $|I|= 1,$ then  $Bm\cap V$ contains \begin{eqnarray}\label{eq3}
(a- q^n)(m_n+v)&=& q^nv_n +
(a-q^n)v,\nonumber
\end{eqnarray} and by Step 3, this is non-zero.
 Similarly if $|I| >1$, then $Bm$ contains
$(a- q^n)m$ and we have $(a- q^n)m = \sum_{j \in J} \mu_j m_j +v'$  with $J = I \backslash \{n\}, v' \in V,$
and $\mu_j \neq 0$ for $j \in J.$  Thus the result follows by induction.
\hfill \epf
%%%%%%%%%%%%%%%%%%%%%%%%%%%%%%%%%%%%%%%%%%%%%%%%%%%
%%%%%%%%%%%%%%%%%%%%%%%%%%%%%%%%%%%%%%%%%%%%%%%%%%%
%\input{qwexample}
\section{The Quantized Weyl Algebra}\label{qwa}
Throughout this section assume that $q$ is an element of $K$ which is not a root of unity. We show that the quantized Weyl algebra $C = C(q)$ does not have property $(\diamond)$.   We begin with some comments which may serve to motivate our construction. Observe that in  Theorem \ref{theoremB}, the submodules of $W = Bw_0$ have the form $Bn^kw_0$ for some normal element $n$ of $B.$ An analogous statement holds for the Example from \cite{CH} mentioned  in the Introduction.  Now the element $n = ab - ba \in C$ is normal, and we can in fact repeat this strategy.  Note however that $n$ has degree two with respect to a natural filtration on $C$, whereas in the earlier examples the normal element had degree one.  For this reason, we have not attempted to give a more unified treatment of our results.
%We construct a finitely generated nonartinian essential extension of a simple $C=C(q)$ -module.

It is reasonable to look for a $C$-module $W$ such that $W = K[n]$ as a $K[n]$-module with $(n^{i})$ a submodule of $W$ for each $i.$  Note that $\bar C = C/Cn \simeq K[a^{\pm 1}]$, and that if such a module $W$ exists, then each factor $(n^{i})/((n^{i+1})$ is a one-dimensional $\bar C$-module.
 Based on these considerations, it is not hard to determine the possibilities for $W$, and with a little experimentation, arrive at the required nonartinian monolithic module.

Consider the $K$-vector space $M$ with basis $\{v_i, w_i: i,j\in
\N\}$, and let $V={\mbox span}_K\{v_i: i\in \N\}$, $W = M/V.$ Define linear operators $a$ and $b$ on $V$ by
\begin{equation}\label{C1}
 av_0=0
\end{equation}
\begin{equation}\label{C2}
 av_n=\displaystyle\frac{q^n-1}{q-1}v_{n-1}
\end{equation}
\begin{equation}\label{C3}
 bv_n=v_{n+1}
\end{equation}
Next extend the action of $a$ and $b$  to $M$ by setting
\begin{equation}\label{C4}
 aw_n=q^n(w_n+w_{n+1})
\end{equation}
and
\begin{equation}\label{C5}
 bw_n=\frac{q^{-n}}{1-q}w_n+(-1)^nv_0.
\end{equation}
\\
We then have
\begin{equation}
(ab-ba)w_n=-\frac{1}{q}w_{n+1},
\end{equation}
\begin{equation}
(ab-qba)w_n=w_n
\end{equation}
It is now easy to see that $M$ is a $C$-module, and  $V$ is a submodule of $M$.

\begin{lemma}\label{c0}
The $C$-module $V$ is simple.
\end{lemma}
\begin{proof}
Since any element of $V$ is of the form $v=a_0v_0+a_1v_1+\ldots+a_nv_n$ for some $a_i\in K$, by equation (\ref{C2}) we deduce that $v_0\in Cv$ for any nonzero $v\in V$. Hence $V$ is simple and also $V=Cv_0$.
\end{proof}

\begin{theorem} \label{c1} \bi \item[{}] \item[{(a)}]  The module $M$ is a non-Artinian essential extension of the simple submodule $V$.  \item[{(b)}] The submodules of $W$ are linearly ordered by inclusion, and are pairwise non-isomorphic.\ei
\end{theorem}
%\begin{prop}
%The module $M$ is a non-Artinian essential extension of the simple
%module $V$.
%\end{prop}
\begin{proof}
First we prove (b). By equation (\ref{C5}) any submodule of $W$ is spanned by a subset of $\{w_n: n\in {\N}_0\}$. For any $n\in\N$ set
$W_n=span\{w_m:m\geq n\}.$ Consideration on the actions of $a$ and $b$ shows that the complete list of non-zero submodules of $W$ is
$$W=W_0\supset W_1\supset W_2\supset\ldots$$
Since $b$ acts as multiplication by $\displaystyle\frac{q^{-n}}{1-q}$ on the unique simple quotient of $W_n$, the proof of (b) is complete.

Next we prove  (a). By Lemma \ref{c0}, $V$ is simple and by (b) $M$ is not Artinian. The rest of the proof consists of three steps.

(i) Given $n\in \N$, by (\ref{C5}),
\begin{equation}
\left(b-\frac{q^{-n}}{1-q}\right)w_n=(-1)^nv_0\in V\cap Cw_n,
\end{equation}
so $Cw_n\cap V\neq 0$.

(ii)
For any $n\in\N$, $C(w_n+v)\cap V\neq 0$.
Indeed
\begin{equation}
(b-\frac{q^{-n}}{1-q})(w_n+v)=(-1)^nv_0+(b-\frac{q^{-n}}{1-q})v.
\end{equation}
So we must show that we can not have $v\in V\backslash\{0\}$ such
that
$$\left(b-\frac{q^{-n}}{1-q}\right)v=(-1)^{n+1}v_0.$$
This follows since if
$v=\lambda_0v_0+\ldots+\lambda_mv_m$, for some $\lambda_0,\ldots,\lambda_m\in K$ with $\lambda_m \neq 0$, then  the coefficient of $v_{m+1}$ in
$(b-\frac{q^{-n}}{1-q})v$ is non-zero.
%$$(-1)^{n+1}v_0=(-\frac{q^{-n}}{1-q})\lambda_0v_0
%+(\lambda_0-\frac{q^{-n}}{1-q}\lambda_1)v_1+\ldots+
%(\lambda_{m-1}-\frac{q^{-n}}{1-q}\lambda_m)v_m+\lambda_mv_{m+1}$$
%can't hold.

(iii) Let $m\in M\backslash V$. We show that $Cm \cap V\neq 0$. This will complete the proof. Without loss of generality we
can write $m=w_n+\lambda_{n-1}w_{n-1}+\ldots+\lambda_0w_0+v$ for
some $v\in V$ and $\lambda_0,\ldots,\lambda_{n-1}\in K$. %Assume that some $\lambda_i\neq 0$.
Then
$\left(b-\frac{q^{-n}}{1-q}\right)m$ is a linear combination of
$w_{n-1},\ldots,w_0,$ and the $v_i$ with $ i\in\N$. Either we are in case (i) or
(ii) or if not, we apply $\left(b-\frac{q^{-k}}{1-q}\right)$ for a suitable
$k$ and repeat the process.
\end{proof}

\section{A Positive Result.}

Let $A=A(\alpha, \beta,\gamma)$ be a down-up algebra and set
$f(x)=x^2-\alpha x-\beta$. Suppose that $f(x)=(x-r)^2$ where $r$
is a primitive $n^{th}$ root of unity. Thus $\alpha=2r$ and
$\beta=-r^2$. The goal of this section is to prove

\begin{thm}\label{rootof1}
A finitely generated essential extension of a simple $A$-module is
Artinian.
\end{thm}

Suppose first that  $char(K)=p$,  and let $Z'=[d^{np}, u^{np}, (du-rud+\frac{\gamma}{r-1})^n]$.  Using \cite[Theorem 4.4]{Hildebrand} and \cite[Lemma 2.2]{Zhao}, it is easy to see that $A$ is finitely generated over the central subalgebra $Z'$. Therefore $A$ is PI and property $(\diamond)$ holds. For the rest of this section we assume that $char(K)=0$.

We denote the Krull dimension of a ring $B$ by ${\Kdim}\,B$.
If $r = \gam = 1,$ then $A$ is isomorphic to the enveloping algebra of the Lie algebra $\fsl$(2),
and  Theorem \ref{rootof1} holds by \cite{Dahlberg89}. The proof depends on the fact that ${\Kdim}\,A = 2,$ and does not immediately adapt to our situation.  A key step in our proof is the fact that a certain localization of $A$ has Krull dimension 2, see Proposition \ref{nelt}.\\
 \\
We establish some preliminaries. By \cite[Corollary 3.2]{CarvalhoLompPusat} we may assume that $r\neq 1$. Hence case 3 of
\cite[\S 1.4]{CarvalhoMusson} holds and we set
$$w_1=(2\beta+\alpha)ud+(\alpha-2)du+2\gamma;$$
$$w_2=2du-2ud$$
so that $\sigma(w_1)=rw_1$ and $\sigma(w_2)=rw_2+w_1$. Set
$w=w_1/2(r-1)=-rud+du+\varepsilon$ where $\varepsilon=\gamma/(r-1)$.

\begin{lem}
$\overline{A}=A/Aw$ is a PI algebra.
\end{lem}
\begin{proof}
Denote the images of $u$ and $d$ in $\overline{A}$ by
$\overline{u}, \overline{d}$, respectively. Then $\overline{A}$ is
generated by $\overline{u}, \overline{d}$ and we have that
$$-r\overline{u}\overline{d}+\overline{d}\overline{u}+\varepsilon
=0.$$ It follows that $\overline{A}$ is isomorphic to a quantized
Weyl algebra if $\gamma\neq 0$ and to the coordinate ring of a
quantum plane if $\gamma=0$. Since $r$ is a primitive $n^{th}$
root of unity for $n>1$, it is well known that these algebras are
PI.
\end{proof}

Recall that given a ring $D$, an automorphism $\sigma$ of $D$ and
a central element $a\in D$, the generalized Weyl algebra
$D(\sigma, a)$ is the ring extension of $D$ generated by $x$ and
$y$, subject to the relations: $xb=\sigma (b)x, by=y\sigma (b),
\text{for all}\; b\in D, yx=a, xy=\sigma (a).$ Noetherian down-up
algebras can be presented as generalized Weyl algebras, see
\cite{KirkmanMussonPassman}.

We need
the following result of Bavula and van Oystaeyen \cite[Theorem
1.2]{BVo}.

\begin{thm}\label{KdimBVo}
Let $R$ be a commutative Noetherian ring with ${\Kdim}\,R=m$ and let
$T=R(\sigma,a)$ be a generalized Weyl algebra. Then ${\Kdim}\,T=m$
unless there is a height $m$ maximal ideal $P$ of $R$ such that
one of the following holds:
\begin{enumerate}
 \item[a)] $\sigma^n(P)=P$, for some $n>0$;
 \item[b)] $a\in \sigma^n(P)$ for infinitely many $n$.
\end{enumerate}
If there is an ideal $P$ as above such that $a)$ or $b)$ holds
then ${\Kdim}\,T=m+1$.
\end{thm}

Given $\lambda_0,\lambda_1\in K$ and $n\in \Z$ there is a unique $\lambda_n\in K$ such that
$$\lambda_n=\alpha\lambda_{n-1}+\beta\lambda_{n-2}+\gamma.$$

For all $n\in \Z$ we have, see \cite[Lemma 2.3]{CarvalhoMusson}
$$\sigma^{-n}(x-\lambda_0)=(x-\lambda_n,y-\lambda_{n+1}).$$

\begin{lem}\label{maximal}
If $M$ is a maximal ideal of $R$ such that $x\in \sigma^n(M)$ for infinitely many $n$, then $\sigma^n(M)=M$ for some $n>0$.
\end{lem}
\begin{proof}
 We can assume that $x\in M$, that is $M=(x-\lambda_0, y-\lambda_1)$ with $\lambda_0=0$. The solution to the recursive relation is then given by
$$\lambda_n=c_1(r^ n-1)+c_2nr^ n$$
for some fixed $c_1, c_2\in K$. If $\lambda_n=0$, then $nc_2=c_1(1-r^ {-n})$, but the right side of this equation can take only finitely many values. Hence $c_2=0$ and the sequence $\{\lambda_n\}$ is periodic. Clearly this gives the result.
\end{proof}

Since $w$ is a normal element of $A$, the set $\{w^ n|n\geq 0\}$ satisfies the Ore condition. We denote by $A_w$, $R_w$ the localizations of $A$ and $R$ with respect to this set.

\begin{prop} \label{nelt} ${\Kdim}\,A_w=2$.
\end{prop}
\begin{proof}
Note that $A_w=R_w(\sigma, \lambda)$ is a generalized Weyl algebra, so by Lemma \ref{maximal} and Theorem \ref{KdimBVo}, we need to show that for any maximal ideal $P$ of $R_w$ and $n>0$ we have $\sigma^n(P)\neq P$. We show that equivalently if $M$ is a maximal ideal of $R$ such that $\sigma^n(M)=M$, then $w\in M$. Indeed if $M=(w_1-a_1,w_2-a_2)$ then from \cite[Lemma 2.2(ii)]{CarvalhoMusson} we have $a_1=0$ and the result follows.
\end{proof}
\par\noindent{\bf Proof of Theorem \ref{rootof1}.} Let $V$ be a simple $A$-module and $M$ a finitely generated essential extension of $V$. There are two cases.

If $wV=0$, it is enough to show that $N=\ann_M(Aw)$ is Artinian. However $N$ is a module over the PI algebra $A/Aw$.

If $wV\neq 0$ then since $w^n$ is central there exists $\lambda\in K$, $\lambda\neq 0$ such that $(w^ n-\lambda)V=0$. By \cite[Theorem 3.15 ]{Praton} $P=(w^ n-\lambda)A$  is prime. By a similar argument as before we can assume $PM=0$. Let $r,s\in K[w]$ be such that
$$1=rw+s(w^ n-\lambda).$$
This implies that $M=wM$ and $ann_M(w)=0$, otherwise $wV=0$. So $M$ is an $A_w$-module which is annihilated by $P_w$. Since %$\Kdim(A_w)=2$ K.dim$(A_w)=2$ \\%
${\Kdim}\,A_w=2$
and $P_w$ is a nonzero prime ideal, $A_w/P_w$ is a prime  of Krull dimension one and the result follows from \cite[Prop 5.5]{M}. {$\square$\par\bigskip}

\section{Down-up Algebras}
{\bf Proof of Theorem \ref{c2}}
If the roots of $X^2-\alpha X-\beta$ are both equal to one or distinct roots of unity it follows from  \cite[Corollary 3.2]{CarvalhoLompPusat} that any finitely generated monolithic $A$-module is Artinian. By Theorem \ref{rootof1}, the same holds if both roots of the quadratic equation are equal roots of unity.

Suppose that the roots of $X^2-\alpha X-\beta$ are not both roots of
unity. Note that 1 is a root of this equation, and in this case the other root is $-\gb.$ By Lemma \ref{homomorphic}, either the coordinate algebra of the quantum plane $B(q)$ or the quantized Weyl algebra $C(q)$ (with $q$ not a root of 1) is a homomorphic image
of $A$ depending on $\gamma =0$ or $\gamma\neq 0$ respectively. Hence by Theorems \ref{theoremB} and \ref{c1} it follows that $A$ does not
satisfy condition $(\diamond)$.
$\square$

\end{document}